\documentclass[12pt]{article}

\usepackage{fullpage}
\usepackage{amsfonts, amsmath, amssymb,latexsym,epsfig}
\usepackage{amsthm}
\usepackage{tikz}
\usetikzlibrary{graphs}
\usetikzlibrary{graphs.standard}
\usepackage{relsize}
\usepackage{verbatim}
\usepackage[shortlabels]{enumitem}
\setlist[enumerate]{label={(\arabic*)}}
\usepackage[skip=10pt plus1pt, indent=20pt]{parskip}

\newtheorem{thm}{Theorem}[section]

\newtheorem{lem}[thm]{Lemma}

\newtheorem{conj}[thm]{Conjecture}

\providecommand{\keywords}[1]{\textbf{\textit{Keywords.}} #1}

\def\W{\mathcal{W}}
\def\FF{\mathbb{F}}
\def\Cay{\operatorname{Cay}}

\usepackage{todonotes}

\usepackage[english]{babel}

\title{Examples of diameter-2 graphs with no triangle or $K_{2,t}$}
\author{
    Sean Eberhard\thanks{Mathematics Institute, Zeeman Building, University of Warwick, Coventry~CV4~7AL, UK.
    \texttt{sean.eberhard@warwick.ac.uk}}
    \and
    Vladislav Taranchuk\thanks{Department of Mathematics: Analysis, Logic and Discrete Mathematics, Ghent University, Belgium. \texttt{Vlad.Taranchuk@UGent.be}}
    \and
    Craig Timmons\thanks{Department of Mathematics and Statistics, California State University Sacramento, 6000 J Street, Sacramento, CA 95819.
    \texttt{craig.timmons@csus.edu}}
}\date{\today}

\begin{document}

\maketitle

\begin{abstract}
    For each $t \ge 1$ let $\W_t$ denote the class of graphs other than stars that have diameter $2$ and contain neither a triangle nor a $K_{2,t}$.
    The famous Hoffman--Singleton Theorem implies that $\W_2$ is finite.
    Recently Wood suggested the study of $\W_t$ for $t > 2$ and conjectured that $\W_t$ is finite for all $t \ge 2$.
    In this note we show that (1) $\W_3$ is infinite, (2) $\W_5$ contains infinitely many regular graphs, and (3) $\W_7$ contains infinitely many Cayley graphs.
    Our $\W_3$ and $\W_5$ examples are based on so-called crooked graphs, first constructed by de Caen, Mathon, and Moorhouse.
    Our $\W_7$ examples are Cayley graphs with vertex set $\FF_p^2$ for prime $p \equiv 11 \pmod {12}$.
    \end{abstract}


\section{Introduction}

Let $G$ be a graph.
The \emph{diameter} of $G$ is the maximum distance between two vertices in $G$.
The \emph{girth} of $G$ is the length of a shortest cycle in $G$, or else $\infty$ if $G$ is acyclic.
The girth of a graph can also be defined in terms of forbidden subgraphs.
Given a family of graphs $\mathcal{F}$, we say $G$ is \emph{$\mathcal{F}$-free} if $G$ has no subgraph isomorphic to a member of $\mathcal{F}$.
Hence, $G$ has girth at least $g$ if and only if $G$ is $\{C_3, C_4, \dots , C_{g-1} \}$-free.

A particularly important family of graphs is the class of diameter-2 graphs of girth 5.
Such a graph $G$ is known as a \emph{Moore graph} (of diameter 2).
The famous Hoffman--Singleton Theorem (see \cite{HS,Singleton}) asserts that $G$ is $d$-regular and of order $d^2+1$ for some $d \in \{2, 3, 7, 57\}$, and in fact the only such graphs are
\begin{enumerate}
    \item the $5$-cycle ($d = 2$),
    \item the Petersen graph ($d = 3$),
    \item the Hoffman--Singleton graph ($d = 7$),
    \item possibly, one or more graphs of degree $d = 57$ and order $3250$.
\end{enumerate}
The existence of a graph of the fourth type has been a mystery for $65$ years.

Attempting to identify a relaxation of the class of Moore graphs that is more suggestive of extremal rather than algebraic combinatorics,
Wood proposed the study of $\{C_3, K_{2,t}\}$-free graphs of diameter $2$ (see \cite{Devillers}).
Let $\W_t$ denote the class of such graphs apart from stars (\emph{Wood graphs}).
Note that $\W_2$ is precisely the class of Moore graphs, so in particular $\W_2$ is finite.
Wood conjectured the following generalization.

\begin{conj}\label{conj:Wood}
    The class $\W_t$ is finite for all $t \ge 2$. In other words, there is an $n_t$ such that if $G$ is a $\{ C_3 , K_{2,t} \}$-free graph of diameter 2 with $n > n_t$ vertices then $G$ is isomorphic to the star graph $K_{n-1,1}$.
\end{conj}

This conjecture and particularly the class $\W_3$ were studied by Devillers, Kam\v{c}ev, McKay, \'O Cath\'{a}in, Royle, Van de Voorde, Wanless, and Wood \cite{Devillers}, who found more than five million graphs in $\W_3$ but no infinite family.

Our main contribution is to present an infinite family of graphs in $\W_3$, which shows that shows this conjecture is false for all $t \ge 3$, albeit for a sparse set of $n$'s.

\begin{thm}\label{MainThm}
    For each integer $e \geq 3$ there is a graph $G \in \W_3$ of order $2^{4e-1} + 2^{2e} + 1$.
\end{thm}

Our proof of Theorem~\ref{MainThm} uses a family of antipodal, distance-regular graphs of diameter $3$ called \emph{crooked graphs},
defined by Bending and Fon-Der-Flaass~\cite{Bending} generalizing previous work of de Caen, Mathon, and Moorhouse~\cite{deCaen}.
We modify these graphs suitably to reduce the diameter to $2$ without creating any triangles or $K_{2,3}$'s.

Notably, the resulting graphs are not regular, and it would be interesting to find an infinite family of regular graphs in $\W_3$.
We have not found such a family, but we can construct an infinite sequence of regular graphs in $\W_5$.
The construction is again based on crooked graphs.

\begin{thm}\label{SecondThm}
    For each integer $e \geq 2$ there is a regular graph $G \in \W_5$ of order $2^{2^{e}-1}$.
\end{thm}

Finally, having the view that a family of vertex-transitive graphs would be even better, we give an unrelated construction of an infinite sequence of Cayley graphs in $\W_7$.
These graphs were first described on the first author's blog~\cite{Eberhard} and arose in discussions between him and Padraig \'O Cath\'ain.

\begin{thm}\label{ThirdThm}
    For every prime $p \equiv 11 \pmod {12}$ there is a symmetric subset $A \subset \FF_p^2$ such that the Cayley graph $G = \Cay(\FF_p^2, A) \in \W_7$.
\end{thm}


\section{Crooked graphs}\label{section 2}

Our proof of Theorem~\ref{MainThm} uses a remarkable family of antipodal distance-regular graphs of diameter $3$ called \emph{crooked graphs} (see~\cite{Bending, deCaen}).
Let $V$ and $W$ be $n$-dimensional vector spaces over $\FF_2$.
A function $Q : V \rightarrow W$ is \emph{crooked} if
\begin{enumerate}
    \item $Q(0) = 0$,
    \item $\sum_{i=1}^4 Q(x_i) \ne 0$ for all distinct $x_1, x_2, x_3, x_4 \in V$ with $x_1 + x_2 + x_3 + x_4 = 0$, and
    \item $\sum_{i=1}^3 (Q(x_i) + Q(x_i + a)) \ne 0$ for all $x_1, x_2, x_3 \in V$ and $a \in V \setminus \{0\}$.
\end{enumerate}
The prototypical example of a crooked function is $Q(x) = x^3$ where $V = W = \FF_{2^n}$ with $n$ odd.
Given a crooked function $Q$, the corresponding \emph{crooked graph} $G_Q$ is the graph $G_Q$ with vertex set $V \times \FF_2 \times W$ where distinct vertices $(a, i, \alpha)$ and $(b, j, \beta)$ are adjacent if and only if
\begin{equation}\label{crooked adjacency}
    \alpha + \beta =  Q(a + b) + (i + j + 1) (Q(a) + Q(b)).
\end{equation}
The following properties were established in \cite{Bending, deCaen} (see particularly \cite[Section~2]{Bending}).%
\footnote{Conversely, Godsil and Roy \cite{Godsil} characterized crooked functions in terms of the distance-regularity of the corresponding graph defined by \eqref{crooked adjacency}.}

\begin{enumerate}
    \item $n$ is odd,
    \item $G_Q$ is distance-regular of order $2q^2$, and degree $2q-1$, and diameter $3$,
    \item $G_Q$ is a triangle-free,
    \item any pair of vertices at distance two have exactly two common neighbors,
    \item $G_Q$ is antipodal, i.e., if $u, v, w \in V(G_Q)$ are distinct vertices and $d(u,v) = d(u, w) = 3$ then $d(v, w) = 3$,
    \item the map $(a, i, \alpha) \mapsto (a, i)$ defines a $q$-fold cover $G_Q \to K_{2q}$ whose fibres $I_1, \dots, I_{2q}$ are the cliques of the distance-$3$ graph and such that the edges between any two distinct fibers form a perfect matching.
\end{enumerate}

\begin{proof}[Proof of Theorem \ref{MainThm}]
    Let $q = 2^{2e-1}$ and let $Q : \FF_q \to \FF_q$ be any crooked function, for example $Q(x) = x^3$.
    Let $G_Q$ be the corresponding crooked graph.
    Define a new graph $G_Q'$ by adding $2q+1$ vertices $v_1, \dots , v_{2q}, v$ to $G_Q$ with the following adjacency rules.
    For $1 \leq j \leq 2q$, the neighborhood of $v_j$ is $I_j \cup \{v\}$.
    The neighborhood of $v$ is $\{v_1, \dots, v_{2q}\}$.
    Then $G_Q'$ has $2q^2 + 2q + 1$ vertices and it contains $G_Q$ as an induced subgraph. We claim that $G'_Q \in \W_3$.

    If $G_Q'$ contains a triangle, then this triangle must contain at least one of the new vertices $v_1, \dots, v_{2q}, v$.
    There is no triangle containing $v$ because the neighborhood of $v$ is the independent set $\{v_1, v_2 , \dots v_{2q} \}$.
    Similarly there is no triangle containing $v_j$ because the neighborhood of $v_j$ is the $I_j \cup \{v\}$, which is also an independent set.
    Therefore $G_Q'$ is triangle-free.

    Next we claim that $G_Q'$ does not contain a $K_{2,3}$.
    Equivalently, any two distinct vertices $x, y \in V(G_Q')$ have at most two common neighbors.
    We may assume $x$ and $y$ are not adjacent, since $G_Q'$ is triangle-free.
    There are a number of cases.
    If $x = v$ and $y \in I_j$ then $x$ and $y$ have a unique common neighbor $v_j$.
    If $x = v_j$ and $y = v_{j'}$ then $x$ and $y$ have a unique common neighbor $v$.
    If $x = v_j$ and $y \in I_{j'}$ then $x$ and $y$ have a unique common neighbor $z \in I_j$, because the edges between $I_j$ and $I_{j'}$ form a perfect matching.
    If $x$ and $y$ lie in a common fiber $I_j$ then they have a unique common neighbor $v_j$.
    Finally, if $x \in I_j$ and $y \in I_{j'}$ with $j \ne j'$ then their common neighbors are exactly their common neighbors in $G_Q$, of which there are exactly two.

    The case analysis of the previous paragraph also shows that every two nonadjacent vertices have at least one common neighbor, so $G_Q'$ has diameter $2$. Thus $G_Q' \in \W_3$.
\end{proof}

We now turn our attention to proving Theorem \ref{SecondThm}, which yields an infinite family of regular graphs $G \in \W_5$.
As in the proof of Theorem \ref{MainThm}, we start with $G_Q$, but instead of adding vertices we add edges within each fiber.


\begin{lem}\label{Lem: Induction}
Let $G_Q$ be a crooked graph on $2q^2$ vertices where
$q = 2^{2e- 1}$ and $e$ is a positive integer.
Let $H \in \W_5$ have order $q$.
If one embeds a copy of $H$ into each fiber of $G_Q$, then the resultant graph is again in $\W_5$.
\end{lem}

\begin{proof}
    Let $G_Q''$ be a graph resulting from embedding a copy of $H$ into each fiber of $G_Q$.
    For each index $j$ let $H_j \cong H$ denote the subgraph of $G_Q''$ induced by $I_j$.
    Note that since no edges were added between any two distinct fibers, the edges between them still form a perfect matching in $G_Q''$.
    Also, $G_Q ''$ has $G_Q$ as a subgraph.

    Since $H_j$ has diameter $2$, any two vertices in the same fiber $H_j$ are at distance at most $2$.
    Any two vertices in distinct fibers are at distance at most 2 in $G_Q$, and so also have distance at most 2 in $G_Q''$.
    This shows $G_Q''$ has diameter 2.

    There is no triangle in $G_Q''$ with all three vertices contained in the same fiber $H_j$, because $H_j$ is triangle-free.
    Similarly, there is no triangle in $G_Q''$ with vertices in distinct fibers, because $G_Q$ is triangle-free and no edges were added between fibers.
    Finally, there is no triangle in $G_Q''$ with one vertex in a fiber $H_i$ and the other two vertices in another fiber $H_j$, because the edges between $H_i$ and $H_j$ form a perfect matching.
    Thus $G_Q''$ is triangle-free.

    Now suppose that $G_Q''$ contains a $K_{2, 5}$ with vertex classes $\{r_1, r_2\}$ and $\{\ell_1,\ell_2,\ell_3,\ell_4,\ell_5\}$.
    We argue similarly as above.
    It cannot be that all $7$ vertices are contained in a single fiber $H_j$, because $H_j$ is $K_{2,5}$-free.
    If $r_1, r_2 \in H_i$ for some $i$, then some $\ell_k$ must be contained in a different fiber $H_j$, but this is impossible because the edges between $H_i$ and $H_j$ form a perfect matching.
    Thus $r_1$ and $r_2$ must be in distinct fibers, say $H_1$ and $H_2$.
    Similarly, $\ell_1, \dots, \ell_5$ must be in distinct fibers.
    Therefore at least $3$ of the vertices $\ell_i$ are in fibers distinct from $H_1$ and $H_2$, say $\ell_3 \in H_3, \ell_4 \in H_4, \ell_5 \in H_5$ without loss of generality.
    But this is impossible because $G_Q$ is $K_{2,3}$-free.
    Thus $G_Q''$ is $K_{2, 5}$-free, which finishes the proof that $G_Q'' \in \W_5$.
\end{proof}

\begin{proof}[Proof of Theorem \ref{SecondThm}]
    We use induction on $e$.
    The base case $e = 2$ is established by observing that $K_{4,4} \in \W_5$.
    Now suppose inductively that $H \in \W_5$ is a regular graph of order $q = 2^{2^e-1}$.
    Let $Q : \FF_q \to \FF_q$ be a crooked function and let $G_Q$ be the corresponding crooked graph of order $2q^2 = 2^{2^{e+1}-1}$.
    Embed copies of $H$ into the fibers of $G_Q$, obtaining $G_Q''$,
    which is again a regular graph.
    By Lemma~\ref{Lem: Induction}, $G_Q''$ is again in $\W_5$.
    This completes the induction.
    \end{proof}

\section{Cayley graphs}

Finally we prove Theorem~\ref{ThirdThm}.

\begin{proof}[Proof of Theorem~\ref{ThirdThm}]
    Let $p \equiv 11 \pmod {12}$ be a prime.
    By quadratic reciprocity, this condition ensures that $-1$ and $-3$ are quadratic nonresidues in $\FF_p$. Let $V$ be the group $\FF_p^2$ where the operation is  component-wise addition.  Define $A \subset V$ by
    \[
        A = \{(x, \pm x^2) : x \in \FF_p \backslash \{ 0 \} \}.
    \]
    Let $G = \Cay(V, A)$.
    By definition, distinct vertices $(x_1, y_1), (x_2, y_2) \in V$ are adjacent if and only if
    \[
        y_1 - y_2 = \pm (x_1 - x_2)^2 \ne 0.
    \]
    We claim that $G$ is $\{C_3, K_{2,7}\}$-free and has diameter $2$.
    Since $G$ is vertex-transitive, it is enough to show that the vertex $(0,0)$ is not in a triangle, and that for any nonzero $(a,b) \in \FF_p^2 \setminus A$, there is at least one but at most six paths of length 2 from $(0,0)$ to $(a,b)$.

    The neighborhood of $(0,0)$ is the set $A$.  Suppose a pair of distinct vertices in $A$ are adjacent.  This implies that there exists $x,y \in \FF_p \backslash \{ 0 \}$ such that one of the following equations holds: $x^2 + y^2 = (x-y)^2$, $x^2 + y^2 = - (x-y)^2$, $x^2 - y^2 = (x- y)^2$, or $x^2 - y^2 = - ( x-y )^2$.
    This first, third, and fourth cannot occur since $x$ and $y$ are distinct and not zero.  The second implies $x^2 - xy + y^2 = 0$ which implies $x/y$ is a solution to the quadratic equation $X^2 - X + 1 = 0$.  This discriminant of this quadratic is $-3$ which is a quadratic non-residue in $\FF_p$.  We conclude that there can be no such $x$ and $y$, so the neighborhood of $(0,0)$ is an independent set.  By vertex-transitivity, $G$ is $C_3$-free.

    We now complete the proof of Theorem \ref{ThirdThm} by showing that for any nonzero $(a,b) \in \FF_p^2 \setminus A$, the number of paths of length 2 from $(0,0)$ to $(a,b)$ is at least one and at most six.  Fix such a vertex $(a,b) \notin A \cup \{(0,0)\}$.  Suppose that $(x,y)$ is the middle vertex on a path of length 2 from $(0,0)$ to $(a,b)$.  Since $(x,y)$ is a neighbor of $(0,0)$, we have that $y = x^2$ or $y = -x^2$.  Since $(x,y)$ is adjacent to $(a,b)$, we have that
    $b - y = (a-x)^2$ or $b- y  = - (a-x)^2$.  Routine calculations show that these four possibilities lead to the four equations $2x^2 - 2ax + (a^2 - b) = 0$, $2x^2 - 2ax + (-a^2 - b) = 0$, $2ax + (b - a^2) = 0$, and
    $2ax + ( -b - a^2 ) = 0$.  The first two equations are quadratics in $x$ and so each has at most two solutions.

    \smallskip
    \noindent
    \textit{Case 1:} If $a  \neq 0$, then the latter two equations are linear in $x$ and so each has a unique solution.
    Moreover, the solution is not $x = 0$ because $(a, b) \notin A$.
    This gives at least one and at most six paths of length 2 from $(0,0)$ to $(a,b)$.

    \smallskip
    \noindent
    \textit{Case 2:} If $a = 0$, then the latter two equations are impossible, otherwise we get $a=b=0$, but $(a,b) \neq (0,0)$.  The quadratic equations are now $b = 2x^2$ and $-b = 2x^2$.  Since $-1$ is a quadratic non-residue, exactly one of these equations has two solutions while the other has none.  Thus, in Case 2 we have exactly two paths of length 2 from $(0,0)$ to $(a,b)$.

    \smallskip
    Once again using the fact that $G$ is vertex-transitive, we can say that $G$ is $K_{2,7}$-free and has diameter 2, so $G \in \W_7$.
\end{proof}

\section{Further remarks}\label{Conclusion}

Crooked graphs were first named by Bending and Fon-Der-Flaass~\cite{Bending} after abstracting the key properties of the crooked functions $Q(x) = x^{2^k+1}$ considered by de Caen, Mathon, and Moorhouse~\cite{deCaen}.
There are now many known crooked functions.
It follows from \cite[Proposition~11]{Bending} that any quadratic almost perfect nonlinear (APN) permutation function is crooked.
In 2008, Budaghyan, Carlet and Leander~\cite{Bud} discovered the first new infinite family of quadratic APN permutations inequivalent to power APN functions.
More recently, Li and Kaleyski \cite{Kaleyski} constructed another infinite class of quadratic APN permutations that appear to be inequivalent to those in previously known classes.

Crooked graphs have found myriad applications both in spectral graph theory~\cite{Jurisic, Koolen, Levit} and extremal graph theory~\cite{deCaen2, Ihringer}.
The following further application was noted to the second author by Sam Mattheus.
Allen, Keevash, Sudakov, and Verstra\"{e}te \cite{Allen} studied Tur\'{a}n problems on forbidding odd cycles together with at least one bipartite graph.  Using an algebraic construction, they proved that $\textup{ex}(n , \{C_3 , K_{2,3} \} ) \geq \frac{1}{ \sqrt{3} }n^{3/2} - n$ whenever $n$ is of the form $n = 3q^2$ where $q$ is a sufficiently large enough prime congruent to 2 modulo 3.  This implies, due to the density of primes, the lower bound $\textup{ex}(n, \{C_3, K_{2,3} \} ) \geq  \left( \frac{1}{ \sqrt{3} }  + o(1) \right)n^{3/2} $.
If $n$ has the special form $2^{4e-1}$ then an improvement follows from considering a crooked graph $G_Q$ with $q = 2^{2e-1}$.
In particular, for infinitely many $n$,
\[
\left( \frac{1}{ \sqrt{2} } + o(1) \right) n^{3/2}
\leq
\textup{ex}(n , \{C_3 , K_{2,3} \} )
\leq
\textup{ex}(n ,  K_{2,3}  )
\leq
\left( \frac{1}{ \sqrt{2} } + o(1) \right) n^{3/2}.
\]
The upper bound holds for all $n$ because of the K\"{o}vari--S\'{o}s--Tur\'{a}n Theorem.


\end{document}